\documentclass[onenumber,plainnumbering,english,11pt]{amsart}
\usepackage[utf8]{inputenc}
\usepackage[square,numbers]{natbib}
\usepackage[mathscr]{eucal}
\usepackage{amssymb,amsmath,amscd,mathrsfs,paralist,amsthm}
\usepackage{mathrsfs}
\usepackage{bm}
\usepackage[T1]{fontenc}

\usepackage{tikz-cd} 
\usepackage{tikz}
\usetikzlibrary{shapes.geometric, arrows, cd}                      
\usepackage{bm}
\usepackage{mathtools, nccmath}
\usepackage{xpatch}
\xpatchcmd{\NCC@ignorepar}{%
\abovedisplayskip\abovedisplayshortskip}
{%
\abovedisplayskip\abovedisplayshortskip%
\belowdisplayskip\belowdisplayshortskip}
{}{}
\usepackage{graphicx}
\usepackage{pgfplots}
\pgfplotsset{compat=1.15}
\usetikzlibrary{positioning, shapes.geometric}
\usepgfplotslibrary{colormaps}
 
\usepackage{tikz-cd}
\usetikzlibrary{arrows, matrix}
\usepackage[all,cmtip]{xy}
\usepackage[colorlinks,linkcolor=black,urlcolor=black]{hyperref}
\hypersetup{
     colorlinks   = true,
     citecolor    = black}
\usepackage{pgfplots}
\pgfplotsset{compat=newest}
\usepackage[font=small,labelsep=none]{caption}
\usetikzlibrary{matrix}
\usetikzlibrary{cd}
\usepackage[colorlinks,linkcolor=black,urlcolor=black]{hyperref}
\hypersetup{
     colorlinks   = true,
     citecolor    = black}
\usepackage[all,cmtip]{xy}
\theoremstyle{plain}
\newtheorem{theorem}{Theorem}[section]

\newtheorem{lemma}[theorem]{Lemma}
\newtheorem{proposition}[theorem]{Proposition}

\theoremstyle{definition}
\newtheorem{definition}[theorem]{Definition}

\theoremstyle{remark}
\newtheorem{remark}[theorem]{Remark}
\makeatletter
\let\amstexbig\big
\def\newbig#1{%
  \ifx#1|%
    \expandafter\@firstoftwo
  \else
    \expandafter\@secondoftwo
  \fi
  {\big@bar}%
  {\amstexbig{#1}}%
}
\AtBeginDocument{\let\big\newbig}
\def\big@bar{\bBigg@{1.1}|}
\makeatother

\numberwithin{equation}{section}
\newcommand{\R}{\mathbb R}
\newcommand{\N}{\mathbb N}

\newcommand{\C}{\mathbb C}

\newcommand{\theoref}[1]{Theorem~\ref{#1}}
\newcommand{\propref}[1]{Proposition~\ref{#1}}

\newcommand{\lemref}[1]{Lemma~\ref{#1}}

\newcommand{\defiref}[1]{Definition~\ref{#1}}
\newcommand{\secref}[1]{Section~\ref{#1}}
\newcommand{\subsecref}[1]{Subsection~\ref{#1}}
\makeatletter
\@namedef{subjclassname@2020}{\textup{2020} Mathematics Subject Classification}
\makeatother
\begin{document}
\title{Holomorphic projective connections on surfaces and osculating spaces}
\author{Oumar Wone}
\address{Oumar Wone}
    \email{wone@chapman.edu}   
\begin{abstract}
We study complex analytic projective connections on surfaces in projective $n$-spaces in terms of the "second" neighborhood of the surface in the ambient space, and in terms of the osculating behavior of the integral curves. We also investigate the action of a remarkable rational transformation on projective connections, and give the geometrical interpretation of joint invariants of a group closely related to the study of equivalence classes of projective connections on surfaces.\end{abstract}
\maketitle
 \tableofcontents
 \section{Introduction}
 A projective structure on a manifold $M$ (real analytic or complex) is an equivalence class of torsion-free connections on the tangent bundle $TM$ (real or holomorphic) of $M$, where two connections $\nabla$ and $\nabla^\prime$ are projectively equivalent, notated $\nabla\sim\nabla^\prime$, when there exists a $1$-form $\omega$ (real or holomorphic) on $M$ such that
 $$\nabla_X^\prime Y=\nabla_XY+\omega(X)Y+\omega(Y)X$$
 for all $X$, $Y\in\mathcal{X}(M)$, \cite{kobayashi2, weyl}. Equivalent connections determine the same geodesic curves, up to reparametrization. We call the set of unparametrized geodesics of a projective structure, a projective connection. In the case of a surface such a projective is characterized by a differential equation of the form \eqref{eqpr11}.
 
 Over the years there has been many different approaches to study of projective connections.
 
 It is in $19^{th}$ century mathematics (differential geometry of surfaces) that one finds the first examples of projective connections. They are in this case given by the geodesics of a Levi-Civita connection on a surface. This approach was expounded in the works of Gauss, Beltrami \cite{beltrami}, and others. It bears the question of when a given projective structure contains a Levi-Civita connection (possibly pseudo-Riemannian). This type of question was first investigated by Beltrami \cite{beltrami}, for the case of surfaces of constant curvature, and later on by Liouville \cite{liouville}, and more recently by \cite{bryant2009, krug2008, eastwood2020}. For a review of the classical approach and period of the topic we refer to \cite{arnold1980, lie1883, tresse1894, liouville1889, liouville18872, liouville18892, olver}.
 
 Besides there is also the approach of Cartan \cite{cartan1924}. Cartan studied in \cite{cartan1924}, see also \cite{sharpe}, what one may term Cartan 'projective connections' (a notion different for what we introduced above). This entails the study of manifolds $M$ locally modeled on the Klein geometry $PGL(n+1)/H=\R P_n$ (or $\C P_n$), via a so-called Cartan connection. Here $PGL(n+1)$ stands for $PGL(n+1,\R)$ or $PGL(n+1,\C)$ according to the case, and $H$ is the subgroup of $PGL(n+1)$ fixing a point of $\R P_n$ (or $\C P_n$). When instead of $H$ one considers its subgroup $H_1$ fixing a line in $\R P_n$ (or $\C P_n$) and a point on this line, one obtains a Cartan geometry modeled on $G/H_1:=PT(\R P_n)$ (or $PT(\C P_n)$). It is this latter model which is relevant in the study of the so-called path geometry, the simplest case of which corresponds locally to a scalar second order ODE, and happens when $n=2$. In this viewpoint of Cartan's the geodesics of the Cartan connection (for the cases at hand) are curves whose development into the homogeneous model are straight lines. We refer to \cite{bryant} for the description of the Cartan geometry and path geometry corresponding to a scalar second order ODE. 
 
 Our initial definition of projective structures and Cartan projective connections are linked through the notion of normal Cartan projective connection. It is proven that to any projective structure on a manifold $M$, there is associated a unique normal Cartan projective connection, and conversely every normal Cartan projective connection is induced by a unique projective structure, see \cite{cartan1924, kobayashi1, kobayashi2}.
 
Meanwhile independently from Cartan, Whitehead \cite{whitehead} and Thomas \cite{thomas1, thomas2}, developed a different perspective on projective structures. Given the datum of a projective structure on a manifold $M$, Thomas and Whitehead build its so-called associated volume bundle $\widetilde{M}$. Granted that a T-W (Thomas-Whitehead) connection is a torsion-free linear connection $\widetilde{\nabla}$ on $\widetilde{M}$, verifying some appropriate properties, see \cite{thomas1, thomas2, whitehead, roberts}. Restricting to the particular family of 'normal' T-W connections \cite{roberts}, one proves that any projective structure on a manifold $M$ is induced by (corresponds to) a unique normal T-W connection on $\widetilde{M}$. For the holomorphic counterpart of the theory of Thomas and Whitehead, we refer to \cite{molzon}. 

The approaches of Cartan, and Thomas, Whitehead are tied together via the so-called theory of tractor calculus: connections and bundles, see \cite{eastwood, cap1}.

Together, and as a kind of synthesis, they led to the birth of the field of parabolic geometry \cite{cap2} and its modern developments, notably the BGG machinery \cite{cap2}.

%(twistor approach): cone in Lane? Need to look at example of Bompiani-Cartan

On the other hand there is the twistorial approach as expounded by Hitchin \cite{hitchin1982}, Lebrun \cite{lebrun1}, see also \cite{luza}, \cite{hwang}, \cite{kamran}, where a (holomorphic) projective structure on a complex manifold $M$ corresponds (when $\dim M=2$) to an embedding of a rational curve in $M$ with normal bundle of degree $1$. But a similar result holds for complex manifolds of higher dimensions \cite{lebrun1}. This twistorial approach is strongly based on deformation techniques, as epitomized by the famous Kodaira deformation theorem \cite{kodaira}.

In this paper we investigate projective connections on surfaces immersed in projective $n$-space according to the behavior of the osculating planes of the integral curves of the projective connections. What happens is that the geometry of the differential equation defining the integral curves of the projective connection is controlled by the second osculating space of the surface. This space denoted below by $S(2,0)$ can be equal only to $\C P_2$, $\C P_3$, $\C P_4$, $\C P_5$. The case $S(2,0)=\C P_2$ corresponds to a plane surface; the case of $S(2,0)=\C P_3$ gives to two sub-cases, a non-developable surface of $\C P_3$ or a developable (non-planar) surface in $\C P_n$, $n\geqslant3$. The surfaces having $S(2,0)=\C P_4$ comprise also two sub-cases, namely surfaces with a conjugate net, and surfaces with a family of asymptotics. Finally the last remaining case is given by the surfaces having $S(2,0)=\C P_5$, which are called general surfaces. What we do in \secref{bompiani} is we take the different cases of $S(2,0)$ sequentially and show how one may obtain the full family of projective connections in each case. We succeed doing so in the significant cases of $S(2,0)=\C P_4$, $S(2,0)=\C P_5$ and in the case of the non-developable surfaces of $\C P_3$. This is done in the \subsecref{bomp1} and \subsecref{bomp3}, resp. \subsecref{bomp4}, resp. \subsecref{wilc}. And this is illustrated by the \theoref{thb2} and \theoref{thb3}, resp. \theoref{thb4}, resp. \theoref{wilc43}. Following this we in \secref{kas}, exploit a simple rational transformation \eqref{eq2} acting on the higher Cauchy data plane, which is simply the plane $\C^2(u^\prime,u^{\prime\prime})$ above a given point $(x,u)$ of $\C^2$, and deduce some consequences on projective connections. See for instance \lemref{lem1}, \propref{prop1}, \propref{elements}. Further in \secref{invkas} we determine the joint invariants of the pseudogroup $\mathscr P$, see equation \eqref{pro4}, acting on the the $n$-th Cartesian product, $n\geqslant4$, of the higher Cauchy data plane. We find the complete set of rational invariants, see \theoref{inv}, and give a geometric interpretation in \theoref{jointinv} of the invariants gotten in \theoref{inv}. This interpretation is arrived to by using some functorial isomorphism between the two data, see \theoref{theo1}. Finally in \subsecref{secexa} we describe the simplest of the geometric invariants described in \theoref{jointinv}.
\section{Geometrical interpretations of projective connections on some surfaces}
\label{bompiani}
In this \secref{bompiani} we investigate projective connections on holomorphic surfaces of $\C P_n$. Before doing so we introduce some notions that we shall need throughout \secref{bompiani}.
\begin{definition}
\label{cil1}
A regular parametric representation (r.p.r.) of an analytic subvariety of $\mathbf V$ of dimension $m$ of $\C P_n$, in the neighborhood of a smooth point $P\in\mathbf V$, is a triple $(\mathscr D,U,f)$, such that
\begin{enumerate}
\item $\mathbf U$ is an open neighborhood of $P$ in $\mathbf V$.
\item $\mathscr D$ is an open subset of $\C^h$, with $h\geqslant m$.
\item $f:\mathscr D\to \mathbf U$ is a holomorphic surjective map of rank $m$ at any point of $\mathscr D$.
\end{enumerate}
\end{definition}
\begin{definition}
\label{cil2}
Let $(\mathscr D,\mathbf U,f)$ be a regular parametric representation of $\mathbf V\subset \C P_n$. A lifting of $(\mathscr D,\mathbf U,f)$ in a system of homogeneous coordinates of $\C P_n$ is a holomorphic map $y:\mathscr D\to\C^{n+1}-\{0\}$ which makes commutative the following diagram
\[\begin{tikzcd}
  \mathscr D \arrow[r, "y"] \arrow[d, "f"]
    & \C^{n+1}\setminus\{0\} \arrow[d, "\psi"] \\
  \mathbf U \subset \mathbf V\arrow[r, "\subset"]
& |[rotate=0]| \C P_n \end{tikzcd}\]
\end{definition}
Let us fix a point $P$ of $\mathbf V$, and choose a r.p.r. $(\mathscr D,\mathbf U,f)$ of $\mathbf V$ in a neighborhood of $P$. We consider a lifting $y:\mathscr D\to\C^{n+1}-\{0\}$ with respect to a system of homogeneous coordinates; further let $u_0\in \mathscr D$ be a point such that $f(u_0)=P$, and $r$ a non negative integer.

Given a multiindex $I=(i_1,i_2,\ldots,i_h)$ we set $|I|:=\sum_{l=1}^hi_l$ and $y_I:=\dfrac{\partial^{|I|}y}{\partial u_1^{i_1}\ldots \partial u_h^{i_h}}$. We introduce the projective subspace $T(r,P,\mathbf V)\subset \C P_n$ generated by the points $<y_{I}(u_0)>$ for all multiindices such that $|I|\leqslant r$. Then one readily sees that $T(r,P,\mathbf V)$ does not depend on the lifting, the choice of the coordinate system and the choice of the r.p.r.
\begin{definition}
The space $O_{s,P,\mathbf V}:=T(s,P,\mathbf V)$ is called the $s$-th osculating space to $\mathbf V$ in $P$. In particular $T(1,P,\mathbf V)=:T(P,\mathbf V)$ is the tangent projective space to $\mathbf V$ at $P$ and $T(0,P,\mathbf V)=P$.
\end{definition}
\begin{definition}
\label{cil3}
Let us consider in $\C P_n$ an analytic variety $\mathbf V$ of dimension $m\leqslant n$, whose parametric vector equation (an r.p.r with $h=m$) is
$$y=y(u_1,\ldots,u_m),$$
in which $u_1,\ldots,u_m$ are the independent parameters. The coordinates $y$ are single-valued analytic functions of the parameters. The element $E_r$ at a point $P$ of a curve $\mathbf C$ immersed in $\C P_n$ ($0\leqslant r\leqslant n$) is defined to be the configuration composed of the point itself and all the osculating spaces, $O_{s,P,\mathbf C}$, $1\leqslant s\leqslant r$, of the curve at $P$. In particular the element $E_0$ at a point of a curve is the point itself, and the element $E_2$ at a point of a curve is the configuration composed of the point itself, the tangent line and the osculating plane of the curve at the point.

The space $S(k,r)$ at a point of a curve on a variety is defined to be the projective space of least dimensionality containing the osculating spaces $O_{k, P,\mathbf C}$, $k>0$ at the point $P$ of all the curves $\mathbf C$ that lie on the variety and pass through the point and have in common at the point the same element $E_r$, $r<k$. In particular $S(k,0)$, $k>0$, is the projective space of least dimensionality containing the osculating spaces $O_{k,P,\mathbf C}$ at the point of all the curves $\mathbf C$ that lie on the variety and pass through the point. It follows (see \cite[p.~390]{lane}) that
$$S(k,0):=\biggl\langle y,y_x,y_u,\ldots,\frac{\partial ^ky}{\partial x^k},\frac{\partial ^ky}{\partial x^{k-1}\partial u}\ldots,\frac{\partial ^ky}{\partial u^k}\biggl\rangle=T(k,P,\mathbf V),$$
when the variety $\mathbf V$ happens to be a surface.
\end{definition}
 \begin{definition}
 A pencil of projective subspaces of $\C P_n$ is a (projective) linear combination of 2 distinct projective subspaces of $\C P_n$ of the same dimension. Two pencils of subspaces $x_1 S_1+x_2 S_2 $ and $x_3S_3+x_4S_4$, with $\left[x_1:x_2\right]$ and $\left[x_3:x_4\right]$ varying in $\C P_1$, and such that $\dim S_1=\dim S_2$ and $\dim S_3=\dim S_4$, are projective equivalent if there exists $\left(\begin{array}{cc}a & b \\c & d\end{array}\right)\in PGL(2,\C)$ such that $\left[x_3:x_4\right]=\left[ax_1+bx_2:cx_1+dx_2\right]$.
 
 Given an analytic surface $\mathbf V\subset \C P_n$ and $P\in V$, the pencil of tangents to $V$ at $P$ is the set of lines $l$ of $\C P_n$ passing through $P$, and tangent $\mathbf V$ at $P$.
 \end{definition}
We now start our study of \eqref{eqpr11} on a surface $\mathbf S$. The interpretation of equation \eqref{eqpr11}: 
$$u^{\prime\prime}=A+Bu^{\prime}+Cu^{\prime2}+Du^{\prime3},\qquad u^\prime=\dfrac{du}{dx},\;u^{\prime\prime}=\dfrac{d^2u}{dx^2}$$
differs strikingly according to the groups with respect to which it is considered. Being an equation of the second order, its interpretation must depend at each point of the surface $\mathbf S$ only on the neighborhood of the second order at this point. The ambient space has nothing to do with equation \eqref{eqpr11} if this space is not determined by the neighborhood of the second order of one point. In fact in the projective geometric interpretation of \eqref{eqpr11}, the only essential distinction comes from the projective characterization of the neighborhood of second order. Thus we have to distinguish the following cases
\begin{equation}
\label{eqb1}
S(2,0)=\begin{cases}
      \C P_2& \text{ plane surface}, \\
     \C P_3 & \text{developable surface (non planar); non-developable surface in $\C P_3$}\\
     \C P_4& \text{surface with a conjugate net; surfaces with a family asymptotic curves}\\
     \C P_5&\text{general surface}.
\end{cases}
\end{equation}
In each of these spaces $S(2,0)$ we suppose that projective geometry holds: it is unnecessary that a projective geometry holds in the ambient space. What we are constructing is a projective geometry of the surface considered independently of the ambient space, or, better, a geometry of the differential equations represented by the surface. In this sense, to interpret geometrically \eqref{eqpr11} means to give a construction of the element $E_2$ satisfying \eqref{eqpr11}, when an element $E_1$ (point and tangent) is given; the construction must be made in the space $S(2,0)$ at the point.
\begin{remark}
Let $y(x,u)$ be the parametric vector equation of a surface $\mathbf S\subset\C P_n$ then $S(2,0)=\C P_2$ if and only if the surface $\mathbf S$ is a plane and $S(2,0)=\C P_3$ if and only if $\mathbf S$ is developable surface in $\C P_n$ or a non-developable surface immersed in $\C P_3$ (see \cite[\S.8 Ch. IV]{lane}).
\end{remark}
\subsection{Projective connections on non-developable surfaces in $\C P_3$}
\label{wilc}
For the notions of projective differential geometry required in this \subsecref{wilc}, we use \cite[Ch. V]{lane}. Let $\mathbf S$ be an analytic surface of $\C P_3$. Let the homogeneous coordinates $y_1,\ldots,y_4$ of a point $y$ in $\C P_3$ be given as single-valued analytic functions of two independent variables $x$, $u$, on a certain domain $\mathscr D$ by equations of the form
$$y_i=y_i(x,u)\qquad(i=1,2,3,4).$$ 
In vector notation these equations can be written as one single equation
\begin{equation}
\label{eqpr1}
y=y(x,u).
\end{equation}
The locus of the point $y$ when $x$, $u$ vary over $\mathscr D$ is, by definition an analytic surface. If we suppose that the coordinates $y_i(x,u)$ are not solutions of any linear homogeneous first-order partial differential equation of the form
\begin{equation}
\label{eqpr2}
\mathcal Ay_x+\mathcal By_u+\mathcal Cy=0
\end{equation}
where subscripts indicate partial differentiation and the coefficients $\mathcal A,\mathcal B,\mathcal C$ are scalar functions of $x,u$ which are not all zero, then the locus of the point $y$, as $x$, $u$ vary over the domain $\mathscr D$, is by definition, a proper analytic surface.
\begin{remark}
This amounts to requiring that the surface $\mathbf S$ does not degenerate to a point or to a curve. 
\end{remark}
The (embedded) tangent plane at a point $y(x,u)$ is the set of points $Y=(Y_1,\ldots,Y_4)$, where $Y$ is a variable point of the plane such that
\begin{equation}
\label{eqpr3}
\det(Y,y,y_x,y_u)=0.
\end{equation}
Let us further consider on the surface $\mathbf S$ a curve $\mathbf C$ whose curvilinear parametric equations are
\begin{equation}
\label{eqpr4}
x=x(t);\quad u=u(t).
\end{equation}
The osculating plane at the point $y$ of the curve $\mathbf C$ is determined by the points $y$, $y^\prime$, $y^{\prime\prime}$, the last two of which are given by the formula
\begin{equation}
\label{eqpr5}
\begin{split}
&y^\prime=y_xx^\prime+y_uu^\prime\\
&y^{\prime\prime}=y_{xx}x^{\prime2}+2y_{xu}x^\prime u^\prime+y_{uu}u^{\prime2}+y_xx^{\prime\prime}+y_uu^{\prime\prime},\\
&y^\prime:=\dfrac{dy}{dt},\,y^{\prime\prime}:=\dfrac{d^2y}{dt^2},\,u^\prime=\dfrac{du}{dt},\,u^{\prime\prime}=\dfrac{d^2u}{dt^2}.
\end{split}
\end{equation}
Hence the equation of the osculating plane takes the form
\begin{equation}
\label{eqpr6}
\det(Y,y,y^\prime,y^{\prime\prime})=0
\end{equation}
in which the point $Y$ is a variable point in the plane.

The curve $\mathbf C$ is an asymptotic curve on the surface $\mathbf S$ provided its (embedded) tangent plane and its osculating plane coincide. This means that the point $y^{\prime\prime}$ which already belongs to the osculating plane \eqref{eqpr6}, is also in the (embedded) tangent plane \eqref{eqpr3}. Thus this can be expressed as
$$\det(y^{\prime\prime},y,y_x,y_u)=0.$$
Substituting in this equation the expressions given for $y^{\prime\prime}$ by the second formula \eqref{eqpr5}, using some well-known properties of determinants, and multiplying by $dt^2$, we obtain the curvilinear differential equation of the asymptotic curves in the following form
\begin{equation}
\label{eqpr7}
Ldx^2+2Mdxdu+Ndu^2=0
\end{equation} 
where the coefficients $L$, $M$, $N$ are the following determinants  
\begin{equation}
\label{eqpr8}
\begin{split}
&L=\det(y_{xx},y,y_x,y_u) \\
&M=\det(y_{xu},y,y_x,y_u)\\
&N=\det(y_{uu},y,y_x,y_u).
\end{split}
\end{equation}
A surface $\mathbf S\subset\C P_3$ is called developable if
$$LN-M^2\equiv 0.$$
For a non-developable surface $\mathbf S\subset \C P_3$ ($LN-M^2\not=0$).

This condition means precisely that the asymptotic curves form a $2$-web. This net is the parametric net ($dxdu=0$) if and only if $L=0$, $N=0$, $M\not=0$, i.e., if and only if, the points $y_{xx}$ and $y_{uu}$ are linearly dependent on the points $y$, $y_x$, $y_u$, while the points $y_{xu}$, $y$, $y_x$, $y_u$ are linearly independent. Therefore the necessary and sufficient conditions that the asymptotic net be the parametric net $dxdu=0$ on a non-developable surface \eqref{eqpr1} in $\C P_3$ are that the surface be an integral surface of a system of two differential equations (this means exactly that the coordinates $y$ of a variable point of the surface satisfy the system \eqref{eqpr9}) of the form
\begin{equation}
\label{eqpr9}
\begin{split}
&y_{xx}+cy+2ay_x+2by_u=0\\
&y_{uu}+c_1y+2a_1y_x+2b_1y_u=0
\end{split}
\end{equation}
whose coefficients are scalar functions of $x$, $u$, and further that the surface be not an integral surface of any equation of the form
\begin{equation}
\label{eqpr10}
\mathcal Ay_{xu}+\mathcal By_x+\mathcal Cy_u+\mathcal Dy=0,
\end{equation}
whose coefficients are scalar functions of $x$, $u$, which are not all zero. Furthermore the coefficients of \eqref{eqpr9} are not arbitrary, since they must satisfy some integrability conditions, which in our case all follow from $(y_{xx})_{uu}=(y_{uu})_{xx}$.

Let us now come back to the main object of our study in this subsection. Let $\mathbf S$ be a non-developable surface given in vector notation by \eqref{eqpr1}. We suppose that its parametric net is formed of asymptotic curves \eqref{eqpr7}. 
%Then, from \eqref{eqpr9}, we have the following system
% \begin{equation*}
%\label{eqpr9}
%\begin{split}
%&y_{xx}+cy+2ay_x+2by_u=0\\
%&y_{uu}+c_1y+2a_1y_x+2b_1y_u=0.
%\end{split}
%\end{equation*}
A curve $\mathbf C$ on $\mathbf S$ represented by the curvilinear equation
$$u=u(x)$$
on the surface $\mathbf S$: $y=y(x,u)$, is called a geodesic curve, provided the function $u(x)$ satisfies an ordinary second order differential equation of the form
\begin{equation}
\label{eqpr11}
u^{\prime\prime}=A+Bu^{\prime}+Cu^{\prime2}+Du^{\prime3},\qquad u^\prime=\dfrac{du}{dx},\;u^{\prime\prime}=\dfrac{d^2u}{dx^2}
\end{equation} 
where the coefficients $A$, $B$, $C$, $D$ are functions of $x,u$. 
%If the curve $\mathbf C$ is regarded as embedded in a one-parameter family of curves whose curvilinear differential is
%\begin{equation}
%\label{eqpr12}
%du-\gamma dx=0,
%\end{equation}
%with $\gamma$ a function of $x,u$, then \eqref{eqpr11} becomes
%\begin{equation}
%\label{eqpr13}
%\gamma^\prime= A+B\gamma+C\gamma^2+D\gamma^3,\qquad\gamma^\prime=\dfrac{d\gamma}{dx}=\gamma_x+\gamma\gamma_u. \end{equation}
For $\mathbf C$ an integral curve of $\eqref{eqpr11}$ on the surface $\mathbf S$ and $P\in \mathbf C$, the coordinates of $P$ may be regarded as functions of $x$, and the equations \eqref{eqpr5} transform into
\begin{equation}
\label{eqpr14}
\begin{split}
&y^\prime=y_x+y_uu^\prime\\
&y^{\prime\prime}=y_{xx}+2y_{xu} u^\prime+y_{uu}u^{\prime2}+y_uu^{\prime\prime},\quad y^\prime:=\dfrac{dy}{dx},\quad y^{\prime\prime}:=\dfrac{d^2y}{dx^2},\quad u^\prime=\dfrac{du}{dx},\quad u^{\prime\prime}=\dfrac{d^2u}{dx^2}.
\end{split}
\end{equation}
Making use of \eqref{eqpr9} we get
\begin{equation}
\label{eqpr15}
y^{\prime\prime}=-(c+c_1u^{\prime2})y-2(a+a_1u^{\prime2})y_x+(-2(b+b_1u^{\prime2})+u^{\prime\prime})y_u+2u^\prime y_{xu}.
\end{equation}
Let us introduce the local coordinates system $y$, $y_x$, $y_u$, $y_{xu}$ in $\C P_3$, i.e. a $2$-parameters projective frame on $\C P_3$, a projective frame for all $(x,u)\in\mathscr D$. This amounts to the fact that the local coordinates shall be proportional to $\alpha_1$, $\alpha_2$, $\alpha_3$, $\alpha_4$, if that point is defined by an expression of the form
$$\alpha_1y+\alpha_2y_x+\alpha_3y_u+\alpha_4y_{xu}.$$
Moreover with this choice of a projective frame at $P$, one defines the asymptotic tangent at $P$ as the projective lines associated to the vector spaces $\langle y,y_u\rangle$ and $\langle y,y_x\rangle$. We can also consider equation \eqref{eqpr11} in the form
\begin{equation}
\label{eqpr35678}
dxd^2u-dud^2x=Adx^3+Bdx^2du+Cdxdu^2+Ddu^3,
\end{equation}
where no parametrization is preferred.

A point of the osculating plane to $\mathbf C$ (see equations \eqref{eqpr5}, \eqref{eqpr14}, \eqref{eqpr15}) has local coordinates 
\begin{equation}
\label{eqpr35679}
\begin{split}
\chi_1 \qquad\text{(arbitrary)},\\
 \chi_2=(d^2x-2(adx^2+a_1du^2))+\rho dx, \\
 \chi_3=(d^2u-2(bdx^2+b_1du^2))+\rho du,\\
  \chi_4=2dxdu,
\end{split}
\end{equation}
where $\rho$ is an arbitrary parameter. The elimination of $dx$, $du$, $d^2u$, $d^2x$, $\rho$, using \eqref{eqpr35678} and \eqref{eqpr35679} gives for the envelope of the osculating planes of the geodesics on the surface \eqref{eqpr1} that pass through $P$, as the cone whose equation is given by
\begin{equation}
\label{eqpr18}
2\chi_2\chi_3\chi_4+(A+2a_1)\chi_2^3-(3B-2b_1)\chi_2^2\chi_3+(3C+2a)\chi_2\chi_3^2-(D-2b)\chi_3^3=0.
\end{equation}

This envelope defines generically a cubic cone. If $D-2b$ or $A+2a_1$ is zero, but not both, it degenerates to a quadric cone. It degenerates to a line if and only if 
$$D-2b=A+2a_1=0.$$ 
Therefore in summary we have
\begin{theorem}
\label{wilc43}
Consider any family of geodesics on a non-developable surface $\mathbf S$. Those which pass through a given point $P$, of $\mathbf S$, form a one-parameter family. Their osculating planes, at $P$, envelop generically a cubic cone, which intersects the (embedded) tangent plane $\Pi_P$ of $\mathbf S$ at $P$ along its asymptotic tangents. Conversely given any system of cubic cones, associated with the given surface in this manner, one may, using \eqref{eqpr18} solve for the four functions, $A$, $B$, $C$, $D$, of $x$ and $u$. Therefore this cubic cone may be regarded by \eqref{eqpr18}, as corresponding to some equation of form \eqref{eqpr11} or \eqref{eqpr35678}. \end{theorem}

\subsection{Interpretation of projective connections on surfaces with a conjugate net}
\label{bomp1}
We begin to study, the interpretation of \eqref{eqpr1} on a surface with a conjugate net, that is, in relation to an equation of Laplace
\begin{equation}
\label{eqb2}
y_{xu}+ay_{x}+by_u+cy=0.
\end{equation}
This case is in a certain sense the simplest of all. For more on equations of type \eqref{eqb2} we refer to \cite[Ch. IV, \S. 7]{lane}.

Let us give at each point $P$ of the surface $\mathbf S$ a plane $\omega_P$ in the space $S(2,0)=\C P_4$ at $P$, not intersecting, except at $P$, the (embedded) tangent plane $\Pi_P$ at $P$ (see \cite[p.~3, Th. VI.]{sempleroth}). We then have a congruence $\{\omega_P\}$ i.e. a $2$-parameter family of planes attached to $\mathbf S$. We have the following
\begin{theorem}
\label{thb1}
The congruence $\{\omega_P\}$ determines on $\mathbf S$ a $2$-parameter system of curves having the property that the $1$-parameter system of osculating planes to these curves at $P$ intersect the plane $\omega_P$ in lines.
\end{theorem}
\begin{proof}
We call such a system a plane system of curves. In fact, given a tangent $t$ at $P$ (which together with $P$ determines an $E_1$), this tangent determines a space $\C P_3=:S(2,1)$ which contains the osculating planes to all the elements $E_2$ having the given $E_1$ in common. This $\C P_3$ cuts $\omega_P$ in a line $l$ (see \cite[p. 3, Th. VI]{sempleroth}); furthermore one sees that the lines $l$, $t$ determine a plane of an element $E_2$, (as it belongs to the $\C P_3=S(2,1)$). Thus we see that given an element $E_1$, an element $E_2$ having the required property is determined. Hence we obtain a $2$-parameter system of curves having the required property.
\end{proof}
 Let us now prove the following
 \begin{theorem}
\label{thb2}
On a surface $\mathbf S$ with a conjugate net, every plane system is represented by a differential equation of the type \eqref{eqb7} (or the equivalent type \eqref{eqpr11}) where the coefficients $A(x,u)$, $B(x,u)$, $C(x,u)$, $D(x,u)$ are determined by the congruence of planes $\{\omega_P\}$ and may coincide with any four arbitrary given functions by a proper choice of the congruence. Conversely, every equation of the type \eqref{eqpr11}, whatever may be the coefficients $A$, $B$, $C$, $D$, determines at each point of the surface a plane and thus a congruence $\{\omega_P\}$ such that the integral curves of \eqref{eqpr11} are the curves of the plane system determined by $\{\omega_P\}$.
\end{theorem}
\begin{proof}
In order to get an analytic representation of these curves let us introduce a system of local coordinates $(\alpha_0,\alpha_1,\alpha_2,\alpha_3,\alpha_4)$ for the point
$$\alpha_0 y+\alpha_1 y_x+\alpha_2y_u+\alpha_3 y_{xx}+\alpha_4y_{uu}$$
belonging to $S(2,0)=\C P_4$ at $P$. Let us call the $\C P_3$ (with coordinates $y_x,y_u,y_{xx},y_{uu})$, or $\alpha_0=0$, the space at infinity of $S(2,0)$ (this is, of course, only a name for the sake of brevity). We can determine $\omega_P$ by the point $P$ and two other points of the space $\C P_3$ at infinity, $(\alpha_1,\alpha_2,\alpha_3,\alpha_4)$, $(\beta_1,\beta_2,\beta_3,\beta_4)$. The osculating plane of a curve $\mathbf C$ of $\mathbf S$ at $P$ is determined (see \eqref{eqpr14}) by $P$ and the two points $y_xdx+y_udu$ and 
$$y_{xx}dx^2+y_{uu}du^2+(d^2x-2adx du)y_x+(d^2u-2bdxdu)y_u$$
whose local coordinates in the $\C P_3$ at infinity are $(dx,du,0,0)$ and 
$$(d^2x-2adxdu,d^2u-2bdxdu,dx^2,du^2).$$
The fact that the two planes $\omega_P$ and $O_{2,\mathbf C,P}$ (osculating plane at $P$) intersect in a line (or that their lines at infinity intersect in a point: recall that $P$ belongs to both planes) is expressed by (see \cite[p. 235, \S. 1.1.]{sempleroth}) the vanishing of the following determinant
\begin{equation}
\label{eqb3}
\left|\begin{array}{cccc}dx & du& 0 & 0 \\d^2x-2adxdu & d^2u-2bdxdu & dx^2 & du^2 \\\alpha_1 & \alpha_2 & \alpha_3 & \alpha_3 \\\beta_1 & \beta_2 & \beta_3 & \beta_4\end{array}\right|=0
\end{equation} 
Using the Plücker coordinates of the two lines at infinity (or of the two planes through $P$),
\begin{equation}
\label{eqb4}
p_{ij}=\alpha_i\beta_j-\alpha_j\beta_i,\qquad(ij=12,13,14,23,42,34)
\end{equation} 
and
\begin{equation}
\label{eqb5}
\begin{split}
&\tau_{12}=dxd^2u-dud^2x+2(adu-bdx)dxdu\\
&\tau_{13}=dx^3;\quad\tau_{14}=dxdu^2;\quad\tau_{23}=dx^2du\\
&\tau_{42}=-du^3;\quad\tau_{34}=0
\end{split}
\end{equation}
the preceding condition \eqref{eqb3} may be written in the form
$$p_{34}\tau_{12}+p_{42}\tau_{13}+p_{23}\tau_{14}+p_{14}\tau_{23}+p_{13}\tau_{42}=0$$
or
\begin{equation}
\label{eqb6}
p_{34}(dxd^2u-dud^2x)+p_{42}dx^3-p_{13}du^3+(p_{23}+2ap_{34})dxdu^2+(p_{14}+2bp_{34})dx^2du=0.
\end{equation}
In this equation \eqref{eqb6}, $p_{34}$ is different from zero as this is the condition which says that the planes $\omega_P$, $\Pi_P$ do not intersect in a line, or equivalently that their lines at infinity do not intersect in a point. Indeed we remind that the (embedded) tangent plane $\Pi_P$ is the projective $2$-plane generated by the points $y$, $y_x$, $y_{u}$ then by \cite[p. 235, \S. 1.1.]{sempleroth}, the lines at infinity of $\Pi_P$ and $\omega_P$ do not intersect if and only if the following determinant does not vanish
\begin{equation}
\label{eqhjlhmgersm}
\left|\begin{array}{cccc}1 & 0& 0 & 0 \\0 & 1 & 0 & 0 \\\alpha_1 & \alpha_2 & \alpha_3 & \alpha_3 \\\beta_1 & \beta_2 & \beta_3 & \beta_4\end{array}\right|\not=0,
\end{equation}
that is 
\begin{equation}
\label{eqlmmlzmqhlh}
p_{34}:=\alpha_3\beta_4-\beta_3\alpha_4\not=0.
\end{equation}
Hence we may always make $p_{34}=1$ and equation \eqref{eqb6} may be written
\begin{equation}
\label{eqb7}
dxd^2u-dud^2x=Adx^3+Bdx^2du+Cdxdu^2+Ddu^3
\end{equation}
where
\begin{equation}
\label{eqb8}
\begin{split}
&A=-p_{42};\quad B=-2b-p_{14};\quad C=-2a-p_{23}\\
&D=p_{13};\quad p_{34}=1.
\end{split}
\end{equation}
Conversely equations \eqref{eqb8} may be solved for the $p_{ij}$ (for $ij\not=12$) and $p_{12}$ is given by the quadratic relation between the Plücker coordinates
$$p_{12}p_{34}+p_{13}p_{42}+p_{14}p_{23}=0$$
that is to say, equation \eqref{eqb7} determines at each point of $\mathbf S$ a plane $\omega_P$ through the point.
\end{proof}
Let us now find the locus of the osculating planes to the integral curves of \eqref{eqb7}. Analytically a point $\chi_0 y+\chi_1 y_x+\chi_2y_u+\chi_3 y_{xx}+\chi_4y_{uu}$ of the osculating plane is described by the local coordinates
\begin{equation}
\label{eqpr56478}
\begin{split}
\chi_0\qquad(arbitrary)\\
\chi_1=d^2x-2adxdu+\rho dx\\
\chi_2=d^2u-2bdxdu+\rho du\\
\chi_3=dx^2\\
\chi_4=du^2,
\end{split}
\end{equation}
with $\rho$ an arbitrary parameter. Eliminating $\rho$, $dx$, $du$, $d^2x$, $d^2u$, using \eqref{eqpr56478} and \eqref{eqb7}, we obtain the locus of the the osculating planes to the integral curves of \eqref{eqb7}, as
\begin{equation}
\label{eqb090909}
(\chi_2-A\chi_3-(C+2a)\chi_4)^2\chi_3=(\chi_1+D\chi_4+(B-2b)\chi_3)^2\chi_4
\end{equation}
\subsection{Interpretation of projective connections on surfaces with a family of asymptotics}
\label{bomp3}
We suppose now that the surface $\mathbf S$ satisfies a "parabolic" equation
\begin{equation}
\label{eqb9}
y_{uu}=ay_x+by_u+cy,
\end{equation}
where $a$, $b$, $c$ are functions of $x$, $u$; the curves $dx=0$ are the asymptotics. We refer to \cite[Ch. IV \S. 6]{lane}.
% In $S(2,0)=\C P_4$ at $P$ and through $P$ we give a plane $\omega$ which does not intersect except at $P$, the tangent plane $\Pi_P$. 

 Let us consider through each point $P\in\mathbf S$, a pencil (projective linear combination) of distinct planes whose axis (or base), not belonging to the (embedded) tangent plane $\Pi_P$, we call $l$. We suppose that this pencil is projectively related to the pencil of tangents at $P$. If, as we suppose, the $\C P_3$ to which this pencil of planes belongs, does not contain $\Pi_P$, it will intersect it in a line \cite[p. 3, Th. VI.]{sempleroth}; that is, there is just (only) one plane $\omega_0$ of the pencil which intersects $\Pi_P$ in a line. We assume a projective transformation exists such that to this plane $\omega_0$ corresponds the asymptotic tangents $u$, $dx=0$. Assuming that this has been done at each point $P\in\mathbf S$, we consider the system of curves defined on the surface by the following property: the osculating plane at $P$ to the curve which passes through $P$ with tangent $t$ must intersect in a line the corresponding projectively related plane of the pencil. It is clear that given the element $E_1(P,t)$ we can always construct the element $E_2$; in fact, the $\C P_3=:S(2,1)$ relative to $t$ intersects the plane of the pencil which corresponds to $t$ in the given projective transformation, in a line $l_1$. The lines $l_1$ and $t$ determine the required osculating plane. 

We have
\begin{theorem}
\label{thb3}
Given a point $P$ of $\mathbf S$ (which possesses a family of asymptotic curves), and a pencil of planes through $P$ projectively related to the pencil of tangents and such that the plane corresponding to the asymptotic tangent is the only plane of the pencil which intersects the (embedded) tangent plane in a line, the system of curves whose osculating planes intersect the corresponding plane of the pencil satisfy an equation \eqref{eqpr11}, and, conversely, each such equation of this type represents such a system of curves.
\end{theorem}
\begin{proof}
Let us find the differential equation satisfied by this system of curves. We firstly introduce the local coordinates system $y$, $y_x$, $y_u$, $y_{xx}$, $y_{xu}$ in $\C P_4$, i.e. a $2$-parameters projective frame on $\C P_4$, a projective frame for all $(x,u)\in\mathscr D$. This means that the local coordinates shall be proportional to $\alpha_0$, $\alpha_1$, $\alpha_2$, $\alpha_3$, $\alpha_4$ if that point is defined by an expression of the form
$$\alpha_0y+\alpha_1y_x+\alpha_2y_u+\alpha_3y_{xx}+\alpha_4y_{xu}.$$The axis of the pencil may be determined by the point $P$ and the point $(\alpha_1,\alpha_2,\alpha_3,\alpha_4)$. A third point which with these determines a plane of the pencil, may describe a point range which is projective to the pencil of tangents (on which $dx/du$ is projective coordinate). The local coordinates of such a point will be of the form
\begin{equation}
\label{eqb10}
\begin{split}
&\beta_1dx+\beta_1^\prime du,\;\;\beta_2dx+\beta_2^\prime du,\\
&\beta_3dx+\beta_3^\prime du,\;\;\beta_4dx+\beta_4^\prime du.
\end{split}
\end{equation}
To the asymptotic tangent $dx=0$ corresponds the point $(\beta_1^\prime,\beta_2^\prime,\beta_3^\prime,\beta_4^\prime)$, on the point range, that is the plane of the $3$ points $P$, $(\alpha_1,\alpha_2,\alpha_3,\alpha_4)$; $(\beta_1^\prime,\beta_2^\prime,\beta_3^\prime,\beta_4^\prime)$. The condition that this plane intersects the (embedded) tangent plane $\Pi_P$ in a line is
$$\alpha_3\beta_4^\prime-\alpha_4\beta_3^\prime=0.$$
Making use of the arbitrariness of the point $\beta^\prime=(\beta_1^\prime,\beta_2^\prime,\beta_3^\prime,\beta_4^\prime)$ we can always make $\beta_3^\prime=\beta_4^\prime=0$; and of course $\alpha_3\beta_4-\alpha_4\beta_3\not=0$, otherwise every plane of the pencil would intersect $\Pi_P$ in a line (reasoning similar as in \eqref{eqhjlhmgersm} and \eqref{eqlmmlzmqhlh}, using the fact that the $2$ generating planes of the pencil in question are the plane generated by $P$, $(\alpha_1,\alpha_2,\alpha_3,\alpha_4)$; $(\beta_1^\prime,\beta_2^\prime,\beta_3^\prime,\beta_4^\prime)$, and the plane generated by $P$, $(\alpha_1,\alpha_2,\alpha_3,\alpha_4)$; $(\beta_1,\beta_2,\beta_3,\beta_4)$).

Introducing now the coordinates $p_{ij}$, $p_{ij}^\prime$ for the lines connecting the point $(\alpha_1,\alpha_2,\alpha_3,\alpha_4)$ with $(\beta_1,\beta_2,\beta_3,\beta_4)$, $(\beta_1^\prime,\beta_2^\prime,\beta_3^\prime,\beta_4^\prime)$. We find the equation of the system of curves in the following determinantal form, which expresses the intersection of the osculating plane and the considered plane of the pencil (or of their lines at infinity) 
\begin{equation}
\label{eqb0989}
\left|\begin{array}{cccc}dx & du& 0 & 0 \\d^2x+adu^2 & d^2u+bdu^2 & dx^2 & 2dxdu \\\alpha_1 & \alpha_2 & \alpha_3 & \alpha_4 \\\beta_1dx+\beta_1^\prime du & \beta_2dx+\beta_2^\prime du & \beta_3dx & \beta_3dx\end{array}\right|=0.
\end{equation}
This gives away from the asymptotic curves $dx=0$
\begin{equation}
\label{eqb11}
\begin{split}
p_{34}(dxd^2u-dud^2x)=-p_{42}dx^3-(2p_{23}+p_{14}+p_{42}^\prime)dx^2du\\-(bp_{34}-2p_{13}+2p_{23}^\prime
+p_{14}^\prime)dxdu^2+(ap_{34}+2p_{13}^\prime)du^3.
\end{split}
\end{equation}
Putting $p_{34}=1$ and comparing the coefficients of this equation with those of equation \eqref{eqpr11} we find
\begin{equation}
\label{eqb12}
A=-p_{42},\;\;B=-(2p_{23}+p_{14}+p_{42}^\prime),\;\;C=2p_{13}-b-(2p_{23}^\prime+p_{14}^\prime),\;\;D=a+2p_{13}^\prime.
\end{equation}
These relations show that our curves satisfy an equation \eqref{eqpr11} and, conversely that, given $A$, $B$, $C$, $D$ arbitrarily, it is always possible to choose $p_{ij}$ and $p_{ij}^\prime$ so that these equations are satisfied.

\end{proof}
\begin{remark}
\label{bprem1}
The pencil and the projective transformation are not completely determined by \eqref{eqpr11}, but there are infinitely many of them. In fact the given geometric configuration at each point depends on seven parameters, while equation \eqref{eqpr11} depends on four coefficients.
\end{remark}
\subsection{Interpretation of projective connections on a general surface}
\label{bomp4}
If $\mathbf S$ is a general surface with respect to the neighborhoods of the second order of its points, the second osculating space at a generic point $P$ is a $\C P_5$ (an introduction to general surfaces is given in \cite[p. 414-425]{lane}). In this $\C P_5$ we take an arbitrary space $\C P_3$, $\Omega_P$, not intersecting the (embedded) tangent plane $\Pi_P$ except at $P$ (possible if $\Omega$ is sufficiently general by \cite[p.~3, Th.~VI.]{sempleroth}); we have then a congruence $\{\Omega_P\}$ of spaces $\C P_3$ associated with $\mathbf S$. 

We say that a $2$-parameter system of curves $\mathbf C$ is a spatial system if the osculating planes $O_{2,P,\mathbf C}$ to the curves $\mathbf C$ of the system at each point $P$ intersect in a line $l_2$, $\Omega_P$, the $\C P_3$ through $P$. Such systems exist. In fact, given an element $E_1$ ($P$ and the tangent $t$) this determines a space $\C P_3=: S(2,1)$ which intersects $\Omega_P$ in a line $l_2$ (see \cite[p.~3, Th.~VI.]{sempleroth}); the plane $(l_2t)$ is therefore the osculating plane to the element $E_2$ of a curve. These elements belong to the $2$-parameter family of curves of the searched for system. 

We prove
\begin{theorem}
\label{thb4}
A spatial system is represented by an equation \eqref{eqpr11}; and conversely, every equation of this type represents an spatial system. The spatial systems, give then, the geometric interpretation of every equation \eqref{eqpr11} on a surface with $S(2,0)=\C P_5$.
\end{theorem}  
\begin{proof}
Let us see that the spatial systems are represented by an equation \eqref{eqpr11}. 

Assuming local coordinates $(\alpha_0,\alpha_1,\alpha_2,\alpha_3,\alpha_4,\alpha_5)$ for a point
$$\alpha_0y+\alpha_1 y_x+\alpha_2 y_u+\alpha_3y_{xx}+\alpha_4y_{xu}+\alpha_5y_{uu},$$
the osculating plane at $P$ of a curve is determined by the point $P$ and the points given in local coordinates by $(0,d^2x,d^2u,dx^2,2dxdu,du^2)$ and $(0,dx,du,0,0,0)$. The space $\Omega_P$ may be determined by $P$ and the following $3$ points $(0,\alpha_1,\alpha_2,\alpha_3,\alpha_4,\alpha_5)$; $(0,\beta_1,\beta_2,\beta_3,\beta_4,\beta_5)$; $(0,\gamma_1,\gamma_2,\gamma_3,\gamma_4,\gamma_5)$; as before we call the $\C P_4$: $\alpha_0=0$ the $4$-space at infinity. The intersection of the osculating plane with $\Omega_P$ (or equivalently of the line at infinity of the osculating plane and the plane at infinity of $\Omega_P$) is determined by the determinental condition
\begin{equation}
\label{eqb13}
\left|\begin{array}{ccccc}d^2x & d^2u & dx^2 & 2dxdu & du^2 \\dx & du & 0 & 0 & 0 \\\alpha_1 & \alpha_2 & \alpha_3 & \alpha_4 & \alpha_5 \\\beta_2 & \beta_2 & \beta_3 & \beta_4 & \beta_5 \\\gamma_1 & \gamma_2& \gamma_3 & \gamma_4 & \gamma_5\end{array}\right|=0.
\end{equation}
Introducing the Grassmann coordinates of the plane of the last three points
$$p_{ikl}=\left|\begin{array}{ccc}\alpha_i & \alpha_k & \alpha_l \\\beta_i & \beta_k & \beta_l \\\gamma_i & \gamma_k & \gamma_l\end{array}\right|,$$
we remark that the condition $p_{345}\not=0$, must hold. Indeed we assumed that $\Omega_P$ do not intersect $\Pi_P$ in a line. Recalling that $\Pi_P$ is the projective plane generated by the points $y$, $y_x$, $y_u$, we see that the line at infinity of $\Pi_P$ intersects the plane at infinity of $\Omega_P$ if and only if the following determinant vanishes
\begin{equation}
\label{eqdhrùjpohdùù}
\left|\begin{array}{ccccc}1 & 0 & 0 & 0 & 0 \\0 & 1 & 0 & 0 & 0 \\\alpha_1 & \alpha_2 & \alpha_3 & \alpha_4 & \alpha_5 \\\beta_2 & \beta_2 & \beta_3 & \beta_4 & \beta_5 \\\gamma_1 & \gamma_2& \gamma_3 & \gamma_4 & \gamma_5\end{array}\right|=0.
\end{equation}
Hence $p_{345}\not=0$ expresses the fact that $\Omega_P$ do not intersect $\Pi_P$ in a line. So we may make $p_{345}=1$.

Therefore the preceding equation \eqref{eqb13} may be written
\begin{equation}
\label{eqb14}
dxd^2u-dud^2x=Adx^3+Bdx^2du+Cdxdu^2+Ddu^3
\end{equation}
where
\begin{equation}
\label{eqb15}
\begin{split}
&p_{345}=1,\; \;A=p_{245}\\
&B=-(2p_{235}+p_{145}),\\
&C=2p_{135}+p_{234},\;\;D=-p_{134}.
\end{split}
\end{equation}
Thence the spatial system satisfies an equation \eqref{eqpr11}. We must now prove the converse, that is, that every equation \eqref{eqpr11} represents a spatial system of curves, or that, given the coefficients $A$, $B$, $C$, $D$ arbitrarily it is always possible to determine the $p_{ikl}$ in order that they satisfy \eqref{eqb15}.

The coordinates $p_{ikl}$ (which are ten homogeneous coordinates not linearly connected) must satisfy three independent quadratic relations (analogous to that satisfied by the Plücker coordinates of a line). We have five relations \eqref{eqb15} between these coordinates and the coefficients $A$, $B$, $C$, $D$, so that there are $2$-parameter sets of coordinates satisfying these equations. In fact, from the expressions of $B$ and $C$ we see that it is impossible to determine the four $p_{ikl}$ which appear in them; two of these $p_{ikl}$ are arbitrary. This arbitrariness corresponds to the geometrical fact that a spatial system determines at a point $P$ not only one $\C P_3$ but a $2$-parameter family of them. 
\end{proof}
Using reasoning similar to the previous \subsecref{wilc}, \subsecref{bomp1} we find that the locus of the osculating planes to the integral curves of \eqref{eqb14} is given by (with previous notations)
\begin{equation}
\label{eqb16}
\begin{split}
&2(\chi_2p_{345}-\chi_3p_{245}+\chi_4p_{235}-\chi_5p_{234})\chi_3=(\chi_1p_{345}-\chi_3p_{145}+\chi_4p_{135}-\chi_5p_{134})\chi_4\\
&(\chi_2p_{345}-\chi_3p_{245}+\chi_4p_{235}-\chi_5p_{234})\chi_4=(\chi_1p_{345}-\chi_3p_{145}+\chi_4p_{135}-\chi_5p_{134})\chi_5\\
&\chi_4^2-4\chi_3\chi_5=0
\end{split}
\end{equation}
\begin{remark}
If a surface $\mathbf S\subset \C P_n$ with parametric vector $y(x,u)$ has $S(2,0)=\C P_2$ then $y$ satisfies a completely integrable system (with scalar coefficients) of three independent second order partial differential equation of the form \cite[Ch. IV p. 112, Ch. IX]{lane}
\begin{equation}
\label{eqb27659}
\begin{split}
y_{xu}&=cy+ay_x+by_u\\
y_{xx}&=py+\alpha y_x+\beta y_u\\
y_{uu}&=qy+ry_x+sy_u.
\end{split}
\end{equation}
The straight lines of $\mathbf S$ then satisfy a differential equation of the projective connection type given by
\begin{equation}
\label{eqb74347}
u^{\prime\prime}=-\beta+(\alpha-2b)u^\prime-(s-2a)u^{\prime2}+ru^{\prime3}.
\end{equation}
Indeed the differential equation of the straight lines of the plane can be calculated by making use of the fact that $y, y^\prime, y^{\prime\prime}$ must satisfy a linear relation when the accents indicate total differentiation with respect to $x$ along a straight line of the plane, as we have \eqref{eqpr5}
\begin{equation*}
%\label{eqb986456}
\begin{split}
&y^\prime=y_x+y_uu^\prime\\
&y^{\prime\prime}=y_{xx}+2y_{xu} u^\prime+y_{uu}u^{\prime2}+y_uu^{\prime\prime},\quad y^\prime:=\dfrac{dy}{dx},\quad y^{\prime\prime}:=\dfrac{d^2y}{dx^2},\quad u^\prime=\dfrac{du}{dx},\quad u^{\prime\prime}=\dfrac{d^2u}{dx^2}.
\end{split}
\end{equation*}
Then \eqref{eqb74347} follows by using \eqref{eqb27659}, that $y$, $y_x$, $y_u$ are linear independent and that one has $\det(y,y^\prime,y^{\prime\prime})=0$. 

In case a surface $\mathbf S\subset\C P_n$ with parametric vector equation $y(x,u)$ is developable then \cite[p. 112]{lane} its parametric vector equation satisfies a completely integrable system (with scalar coefficients)
\begin{equation}
\label{eq87875799}
\begin{split}
y_{xu}&=\beta y+ay_x+by_u\\
y_{xx}&=py+\alpha y_x
\end{split}
\end{equation}
\end{remark}
%\subsection{Geometric interpretation of second order ODEs on a general surface}
\section{"Metric study" of projective connections: The transformation of centres}
 \label{kas}
 In this \secref{kas} we study the effect of the rational transformation \eqref{eq2} on points of the higher Cauchy data plane $\C^2(u^\prime,u^{\prime\prime})$, above a fixed point $(x,u)\in\C^2$ and deduce some consequences. 
 
 Let 
 \begin{equation}
 \label{eq1}
 \{X=\phi(x,u), \,U=\psi(x,u),\,\phi(x,u), \,\psi(x,u)\in\mathcal{O}:=\C\{x,u\},\,J:=\phi_x\psi_u-\phi_u\psi_x\not=0\}
 \end{equation}
 be the pseudo-group $\mathscr{P}$ of point transformations of the plane $(\C^2,0)$. 
Let $u^\prime$ and $u^{\prime\prime}$ be jet coordinates on $J_2(\mathcal O)$, the second order jet space. Recall that the action of $\mathscr{P}$ on $(u^\prime,u^{\prime\prime})$ is given by
 \begin{equation}
 \label{pro3}
 \begin{split}
 &U^\prime=\dfrac{\psi_x+u^\prime\psi_u}{\phi_x+u^\prime\phi_u}\\
 &U^{\prime\prime}=\dfrac{\lambda+\mu u^\prime+\nu u^{\prime2}+\xi u^{\prime3}+Ju^{\prime\prime}}{(\phi_x+u^\prime\phi_u)^3},
 \end{split}
 \end{equation}
 where $\lambda$, $\mu$, $\nu$, $\xi$ belong to $\mathcal{O}$ and depend on derivatives of $\phi$ and $\psi$ up to order $2$. Since our study is near a given point $P\in (\C^2,0)$ and we are only interested in the variations of the tuple $(u^\prime,u^{\prime\prime})$, we consider $\lambda$, $\mu$, $\nu$, $\xi$ as well as $a:=\psi_u$, $b:=\psi_x$, $c:=\phi_u$ and $d:=\phi_x$ as numbers. In order to simplify our expressions and calculations we introduce the notation
 \begin{equation}
 \label{nota}
 v:=u^\prime,\,\;w=u^{\prime\prime};\,\;V:=U^\prime,\,\;W:=U^{\prime\prime}.
 \end{equation}
 Thus the transformation \eqref{pro3} takes the form
 \begin{equation}
 \label{pro4}
 \begin{split}
 V=\dfrac{av+b}{cv+d},\,\;W=\dfrac{\lambda+\mu v+\nu v^2+\xi v^3+(ad-bc)w}{(cv+d)^3}.
 \end{split}
 \end{equation}
 Hence the pseudo-group $\mathscr{P}$ induces on the tuple $(v,w)$ the eight-parameter (local) Lie group of transformations: \eqref{pro4} whose infinitesimal generators are given by
 \begin{equation}
 \label{pro5}
 \dfrac{\partial}{\partial v},\,\,\dfrac{\partial}{\partial w},\,\,v\dfrac{\partial}{\partial v},\,\,w\dfrac{\partial}{\partial w},\,\,v^2\dfrac{\partial}{\partial v}+3vw\dfrac{\partial}{\partial w},\,v\dfrac{\partial}{\partial w},v^2\dfrac{\partial}{\partial w},\,v^3\dfrac{\partial}{\partial w}.
 \end{equation} 
Equation\eqref{pro5} defines the Lie algebra of $GL(2,\C)\ltimes S^3\C^2$, where $S^3\C^2\simeq\C^4$ is the third symmetric product of $\C^2$ with itself, see \cite[p.~341]{olver1}.
 \begin{equation}
 \label{lfqhfzhlzdlh}
 \left[z_1:z_2:z_3:z_4:z_5\right]=\left[1:u^\prime:u^{\prime2}:u^{\prime3}:u^{\prime\prime}\right]=:\left[1:v:v^2:v^3:w\right].
 \end{equation}
 This gives and embedding of $(v,w)$ in $\C P_4$. The image of $(v,w)$ lies inside the surface $\mathscr S_3$ of degree $3$
  \begin{equation}
\label{hofkfhjuhf}
 z_1z_3-z_2^2=0,\quad z_2z_3-z_1z_4=0,\quad z^2_3-z_2z_4=0.
\end{equation}
\begin{definition}[\cite{sempleroth}]
\label{cone}
Let $n\in\N_{\geq2}$ and $0\leq h\leq n-2$. Any algebraic manifold in $\C P_n$ which is generated by a variable $\C P_{h+1}$ through a fixed $\C P_h$ is called a $\C P_h$-cone of $\C P_n$. To generate a $\C P_h$-cone, the general procedure is to join a fixed $\C P_h$ to every point of a fixed $k$-dimensional manifold $V_k$ $(k\leq n-h-2)$: the joining $\C P_{h+1}$'s (the projective spaces of minimal dimension containing the fixed $\C P_h$ and the point of $V_k$, see \cite[p.~4, Th. VII.]{sempleroth}) generate a $\C P_h$-cone which we denote by $\C P_h(V_k)$, and we call $\C P_h$ the vertex of this cone, $V_k$ its directrix, and the $\C P_{h+1}$'s its generators. The dimension of $\C P_h(V_k)$ is in general $h+k+1$.
\end{definition}
\begin{remark}[\cite{sempleroth}]
With notation as in the previous \defiref{cone} if the fixed $\C P_h$ does not meet the directrix $V_k$, then the degree of $\C P_h(V_k)$ is equal to the degree of $V_k$. Given a point $p\not\in V_k$, the associated point cone $\C P_0(V_k)$ is the union of all projective lines through $p$ which meet $V_k$.
\end{remark}
Hence $\mathscr S_3$ is a $\C P_0$-cone (point-cone) in $\C P_4$ with vertex $\left[0:0:0:0:1\right]$, and directrix the twisted cubic, see \eqref{lfqhfzhlzdlh}. One may show (see \cite[eq.~3.8]{wone}) that the group of transformations $G$ acting on $\left[z_1:z_2:z_3:z_4:z_5\right]$ and given by
\begin{equation}
\label{eqhrhrttnr}
 \begin{split}
& Z_1=d^3z_1+3cd^2z_2+3c^2dz_3+c^3z_4\\
 &Z_2=bd^2z_1+b(bc+2ad)z_2+a(ad+2bc)z_3+a^2cz_4\\
 &Z_3=b^2dz_1+d(ad+2bc)z_2+c(bc+2ad)z_3+ac^2z_4\\
 &Z_4=b^3z_1+3ab^2z_2+3a^2bz_3+a^3z_4\\
 &Z_5=\lambda z_1+\mu z_2+\nu z_3+\xi z_4+(ad-bc)z_5
 \end{split}
\end{equation} 
is the largest transformation group preserving $\mathscr S_3$. Let $\C^2(u^\prime,u^{\prime\prime})$ be the higher Cauchy data plane, i.e. the plane $\C^2(u^\prime,u^{\prime\prime})$ above a given point $(x,u)$ of $\C^2$. We have
\begin{theorem}[\cite{wone}]
\label{theo1}
There is an equivalence of categories between the geometry of curves (resp. points) in the higher Cauchy data plane $\C^2(u^\prime,u^{\prime\prime})$ with respect to the restriction of the action of the pseudogroup $\mathscr P$ and the projective geometry on the cubic cone in $\C P_4$, i.e. the geometry of curves (resp. points) in $\mathscr S_3$ with respect to the group $G$.\end{theorem} 
\begin{definition}[degree]
The degree of a differential equation polynomial in $u^\prime$, $u^{\prime\prime}$ with coefficients in $\mathcal O$ is equal to the degree of the homogeneous polynomial defining the corresponding $2$-parameters curve of $\mathscr S_3$ in the variables $(z_i)_{1\leqslant i\leqslant5}$.
\end{definition}
To interpret the equations of various degrees geometrically without leaving the plane, consider for an element $(v,w)$ the corresponding "centre of curvature", that is the rational transformation given by
\begin{equation}
\label{eq2}
x_0=-\dfrac{v(1+v^2)}{w};\quad y_0=\dfrac{1+v^2}{w}.
\end{equation}
A point transformation, operating on the element $(v,w)$, according to \eqref{pro4}, induces the following transformation of the centres
\begin{equation}
\label{eq3}
\begin{split}
&X_0:=-\dfrac{V(1+V^2)}{W}=\dfrac{(ax_0-by_0)\{(ax_0-by_0)^2+(cx_0-dy_0)^2\}}{\lambda y_0^3-\mu x_0y_0^2+\nu x_0^2y_0-\xi x_0^3+(ac-bd)(x_0^2+y_0^2)}\\
&Y_0:=\dfrac{1+V^2}{W}=\dfrac{-(cx_0-dy_0)\{(ax_0-by_0)^2+(cx_0-dy_0)^2)\}}{\lambda y_0^3-\mu x_0y_0^2+\nu x_0^2y_0-\xi x_0^3+(ac-bd)(x_0^2+y_0^2)}.
\end{split}
\end{equation}
Hence under arbitrary point transformations the centres undergo a group of cubic transformations, according to \eqref{eq3}. Furthermore using equation \eqref{eq2}, we see that an equation of first rank 
$$Eu^{\prime\prime}=Du^{\prime3}+Cu^{\prime2}+Bu^{\prime}+A,\;\; A,\,B,\,C,\,D,\,E\in\mathcal O,$$
establishes a relation between $x_0$ and $y_0$ of the form
\begin{equation}
\label{eq4}
E(x_0^2+y_0^2)=-Ax_0^3+Bx_0^2y_0-Cx_0y_0^2+Dy_0^3.
\end{equation}
Hence
\begin{lemma}
\label{lem1}For a first degree equation (projective connection), the locus of the centres of curvature of the $1$-parameter of integral curves passing through a fixed point is a cubic curve of the type \eqref{eq4}. 
\end{lemma}
As is well-known, the differential equation of the geodesics of any surface expressed in arbitrary coordinates is of first rank. Therefore
\begin{proposition}
\label{prop1}
If a germ of surface is mapped biholomorphically onto a germ of plane, the geodesics through a point are given by curves whose centres of curvature at the common point lie on a cubic curve of the type \eqref{eq4}.
\end{proposition}
We also have the following 
\begin{proposition}
\label{elements}
Given Any four generic tuples $(v_1,w_1)=(u_1^\prime,u_1^{\prime\prime})$, $(v_2,w_2)=(u_2^{\prime},u_2^{\prime\prime})$, $(v_3,w_3)=(u_3^{\prime},u_3^{\prime\prime})$, $(v_4,w_4)=(u_4^{\prime},u_4^{\prime\prime})$ at the point $(x,u)$, of the higher Cauchy data plane $\C^2(u^\prime,u^{\prime\prime})$, they determine an equation of first rank; hence the four centres of curvature generate a cubic of type \eqref{eq4}.
\end{proposition}
\begin{proof}
 In fact let us consider four generic tuples $(v_1,w_1)=(u_1^\prime,u_1^{\prime\prime})$, $(v_2,w_2)=(u_2^{\prime},u_2^{\prime\prime})$, $(v_3,w_3)=(u_3^{\prime},u_3^{\prime\prime})$, $(v_4,w_4)=(u_4^{\prime},u_4^{\prime\prime})$ at the point $(x,u)$, and an underdetermined equation of rank one in $(v,w)$ coordinates
 $$u^{\prime\prime}=Du^{\prime3}+Cu^{\prime2}+Bu^{\prime}+A.$$
This gives us the system in $A$, $B$, $C$, $D$
   \begin{equation}
\label{mckernan34}
\begin{split}
u_1^{\prime\prime}&=Du^{\prime3}_1+Cu^{\prime2}_1+Bu^{\prime}_1+A\\
u_2^{\prime\prime}&=Du^{\prime3}_2+Cu^{\prime2}_2+Bu^{\prime}_2+A\\
u_3^{\prime\prime}&=Du^{\prime3}_3+Cu^{\prime2}_3+Bu^{\prime}_3+A\\
u_4^{\prime\prime}&=Du^{\prime3}_4+Cu^{\prime2}_4+Bu^{\prime}_4+A.
\end{split}
\end{equation}
We may find $A$, $B$, $C$, $D$ when, for instance, $u_i^\prime\not=u_j^\prime$ for $i\not=j$.
\end{proof}
For an equation of rank $2$
\begin{equation}
\label{pro10}
\begin{split}
&A_0u^{\prime\prime2}+(B_0+B_1u^\prime+B_2u^{\prime2}+B_3u^{\prime3})u^{\prime\prime}\\&+(C_0+C_1u^\prime+C_2u^{\prime2}+C_3u^{\prime3}+C_4u^{\prime4}+C_5u^{\prime5}+C_6u^{\prime6})=0
\end{split}
\end{equation}
where the coefficients belong to $\mathcal{O}$, the central locus of the $2$-parameter family of integral curves through a fixed point is a special sextic 
\begin{equation}
\label{eq5}
\begin{split}
&A_0(x_0^2+y_0^2)^2+(B_0y_0^3-B_1x_0y_0^2+B_2x_0^2y_0-B_3 x_0^3)(x_0^2+y_0^2)+(C_0y_0^6-C_1x_0y_0^5\\
&+C_2x_0^2y_0^4-C_3x_0^3y_0^3+C_4x_0^4y_0^2-C_5x_0^5y_0+C_6x_0^6)=0.
\end{split}
\end{equation}
This sextic degenerates non-trivially when
$$C_3=C_1+C_5,\quad C_0+C_4=C_2+C_6$$
to become a quartic. Further degeneration occurs when $B_0=B_2$, $B_1=B_3$, $C_1=C_5$, and $C_0+2C_6=C_4$. In this case the central locus is the conic
\begin{equation}
\label{eq6}
A_0+B_0y_0-B_1x_0+C_6x_0^2-C_5x_0y_0+C_0y_0^2=0.
\end{equation}
For the general equation of rank $r$, the central locus is a special curve of degree $3r$.
\section{Joint absolute invariants and geometric interpretation}
\label{invkas}
 In this \secref{invkas} we study joint rational absolute invariants of the restriction of the action of $\mathscr P$ (see \eqref{pro4}) on the higher Cauchy plane $\C^2(u^\prime,u^{\prime})$ and give a geometric interpretation of our invariants in terms of the geometry of points on the point-cone over the twisted cubic curve, and finish with an example.
\subsection{Joint absolute invariants}In this subsection we want to study rational expressions depending on $n\geq4$ pairs of given jet coordinates $(u_1^\prime,u_1^{\prime\prime})$, $(u_2^\prime,u_2^{\prime\prime})$, $\ldots,$ $(u_n^\prime,u_n^{\prime\prime})$, and which are invariant under the joint action of the group $\mathscr P$ of \eqref{pro4} and \eqref{pro5}. More precisely this means we want to study the rational absolute invariants of the restriction of the induced action of $\mathscr P$ on the n-th power of the higher Cauchy data plane given by
\begin{equation}
\label{ztmohlzr}
\begin{split}
&\mathscr P\times \overbrace{\C^2\times\C^2\times\ldots\times\C^2}^\text{n times}\\
&(g,((u_1^\prime,u_1^{\prime\prime})\times\ldots\times(u_n^\prime,u_n^{\prime\prime})))\mapsto \biggl(\left(\dfrac{au_1^\prime+b}{cu_1^\prime+d},\,\;\dfrac{\lambda+\mu u_1^\prime+\nu u_1^{\prime2}+\xi u_1^{\prime3}+(ad-bc)u_1^{\prime\prime}}{(cu_1^\prime+d)^3}\right),\ldots,\\
&\left(\dfrac{au_n^\prime+b}{cu_n^\prime+d},\,\;\dfrac{\lambda+\mu u_n^\prime+\nu u_n^{\prime2}+\xi u_n^{\prime3}+(ad-bc)u_n^{\prime\prime}}{(cu_n^\prime+d)^3}\right)\biggr)
\end{split}
\end{equation}
As previously we introduce $u_1^\prime=:v_1$, $u_1^{\prime\prime}=:w_1$, etc in order to simplify the expressions. Let us then consider any number of elements
\begin{equation}
(v_1,w_1),\;(v_2,w_2),\;(v_3,w_3),\ldots,(v_n,w_n).
\end{equation}
A function $f$ of these $2n$ arbitrary given quantities is an absolute invariant for all point transformations $\mathscr P$, \eqref{eq1}, provided it is invariant under the restriction of the induced action of $\mathscr P$, or, equivalently under the eight infinitesimal generators given by \eqref{pro5}. We thus see that such an invariant must satisfy the following system of linear partial differential equations
 \begin{equation}
 \label{eq7}
 \begin{split}
 &\sum_{k=1}^n\dfrac{\partial f}{\partial v_k}=0;\quad\sum_{k=1}^n\dfrac{\partial f}{\partial w_k}=0;\quad\sum_{k=1}^nv_k\dfrac{\partial f}{\partial v_k}=0;\quad\sum_{k=1}^nw_k\dfrac{\partial f}{\partial w_k}=0\\
& \sum_{k=1}^n\left(v_k^2\dfrac{\partial f}{\partial v_k}+3v_kw_k\dfrac{\partial f}{\partial w_k}\right)=0\\
&\sum_{k=1}^nv_k\dfrac{\partial f}{\partial w_k}=0;\quad\sum_{k=1}^nv_k^2\dfrac{\partial f}{\partial w_k}=0;\quad\sum_{k=1}^nv_k^3\dfrac{\partial f}{\partial w_k}=0.
 \end{split}
 \end{equation}
 The first two of the equations \eqref{eq7} show that $f$ involves only the differences of the $v$'s and the $w$'s; while the next two show that $f$ is homogeneous in these differences. Hence $f$ is expressible in terms of $2n-4$ quantities
 \begin{equation}
 \label{eq8}
 \xi_i=\dfrac{v_i-v_1}{v_2-v_1},\quad\eta_i=\dfrac{w_i-w_1}{w_2-w_1},\quad i=3,4,\ldots,n.
 \end{equation} 
 If we introduce these as the new independent variables, the last four equations of the system \eqref{eq7} become
 \begin{equation}
 \label{eq9}
 \begin{split}
 &\sum_{i=3}^n\xi_i(\xi_1-1)+3\sum_{i=3}^n\eta_i(\xi_i-1)\dfrac{\partial f}{\partial \eta_i}=0,\\
& \sum_{3}^n(\eta_i-\xi_i)\dfrac{\partial f}{\partial \eta_i}=0,\quad\sum_{i=3}^n\xi_i(\xi_i-1)\dfrac{\partial f}{\partial \eta_i}=0;\\
&\sum_{i=3}^n\xi_i^2(\xi_i-1)\dfrac{\partial f}{\partial \eta_i}=0.
 \end{split}
 \end{equation}
 If we integrate the first of these equations \eqref{eq9}, we find that $f$ is a function of the $2n-5$ expressions
 \begin{equation}
 \label{eq10}
 \begin{split}
 &r_j=\dfrac{\xi_3(\xi_i-1)}{\xi_j(\xi_3-1)},\quad j=4,\ldots,n\\
 & s_i=\dfrac{\eta_i}{\xi_i^3},\quad i=3,\ldots,n.
 \end{split}
 \end{equation}
 The remaining equations \eqref{eq9} then become
 \begin{equation}
 \label{eq11}
 \begin{split}
& \sum_{i=4}^n(s_i-1)\dfrac{\partial f}{\partial s_i}=0\\
&\dfrac{\partial f}{\partial s_3}+\sum_{i=4}^nr_i\dfrac{\partial f}{\partial s_i}=0\\
&\dfrac{\partial f}{\partial s_3}+\sum_{i=4}^nr_i^2\dfrac{\partial f}{\partial s_i}=0.
 \end{split}
 \end{equation}
 The first of the equations \eqref{eq11} shows that we may take as the independent variables the $2n-6$ quantities
 \begin{equation}
 \label{eq12}
 \sigma_i=\dfrac{s_i-1}{s_3-1},\quad r_i,\qquad i=4,\ldots,n.
 \end{equation}
 The other equations then reduce to
 \begin{equation}
 \label{eq13}
 \begin{split}
& \sum_{i=4}^n(\sigma_i-r_i)\dfrac{\partial f}{\partial s_i}=0\\
&\sum_{i=4}^n(\sigma_i-r_i^2)\dfrac{\partial f}{\partial s_i}=0.
 \end{split}
 \end{equation}
 The next step in the solution brings in the independent variables
 \begin{equation}
 \label{eq14}
 \tau_i=\dfrac{\sigma_i-r_i}{\sigma_4-r_4},\quad r_i,\qquad i=4,\ldots,n,
 \end{equation}
 and reduces the system to the single equation
 \begin{equation}
 \label{eq15}
 \sum_{i=5}^n(r_i-r_i^2+\tau_i(r_4^2-r_4))\dfrac{\partial f}{\partial \tau_i}=0.
 \end{equation}
 Integrating this equation we find that $f$ is a function of
 \begin{equation}
 \label{eq16}
 \begin{split}
& \Omega_l=\dfrac{(r_4^2-r_4)\tau_l-(r_l^2-r_l)}{(r_4^2-r_4)\tau_5-(r_5^2-r_5)},\quad l=6,\ldots,n\\
 &r_i,\quad i=4,\ldots,n.
 \end{split}
 \end{equation}
 In conclusion we have the
 \begin{theorem}
 \label{inv}
 All (absolute) invariants, with respect to the the pseudogroup $\mathscr P$ of point transformations, \eqref{eq1}, of $n\geqslant4$ arbitrary differential elements of the second order $(v_1,w_1)$, $(v_2,w_2)$, $(v_3,w_3)$, $\ldots(v_n,w_n)$, may be expressed as functions of $2n-8$ invariants, namely: $n-3$ cross-ratios
 $$r_4,r_5,\ldots,r_n,$$
 depending only upon the first derivatives (directions) of the elements, and $n-5$ of a new type
 $$\Omega_6,\Omega_7,\ldots,\Omega_n,$$
 which involve the second derivatives (curvature). The new type appears only when there are six (or more) elements. This system of invariants is functionally complete. Furthermore, all rational invariants, i.e., all invariants which involve $v_1$, $w_1$, $\ldots,v_n$, $w_n$ rationally, can be expressed as rational functions of the fundamental invariants $(r_j)_{4\leqslant n}$ and $(\Omega_l)_{6\leqslant l\leqslant n}$.
 \end{theorem}
 Using the functorial isomorphism in the \theoref{theo1} we have the
 \begin{theorem}
 \label{jointinv}
 There is a bijective correspondence between the joint absolute rational invariants with respect to the action of the restriction of $\mathscr P$ to the product of the higher Cauchy data planes $\overbrace{\C^2\times\C^2\times\ldots\times\C^2}^\text{n times}$ given in \theoref{inv}, and the geometric joint absolute invariants with respect to the action of $G$ on $\overbrace{\mathscr S_3\times\mathscr S_3\times\ldots\times\mathscr S_3}^\text{n times}$.
 \end{theorem}
 \subsection{Geometric interpretation of the some of the invariants: Example} 
 \label{secexa}
 \begin{definition}
 \label{mckernan}
 Let $S\subset \C P_n$, $n\geq1$, be a non-empty set of points. We say that $S$ is in linearly general position if any subset of $1\leq k\leq n$ points of $S$ spans a $(k-1)$-plane $\Lambda$.
 \end{definition}
 We have the
 \begin{lemma}
 \label{mckernan1}
 Any two sequences $p_0,p_1,\ldots,p_{n+1}$ and $q_0,q_1,\ldots,q_{n+1}$ of $n+2$ points in linear general position in $\C P_n$ are projectively equivalent, i.e., there is an element $g\in PGL(n+1,\C)$ such that $g(p_i)=q_i$. Furthermore $g$ is unique.
 \end{lemma}
 In order to obtain a geometric interpretation for our absolute invariants, we make use of the corresponding projective geometry in $\C P_4$. If the group considered is the general projective group of $\C P_4$, i.e. $PGL(5,\C)$, and points are taken in arbitrary positions in $\C P_4$, then no absolute invariant occurs for fewer than $7$ points. This is because by using \defiref{mckernan} and \lemref{mckernan1}, any given $6$ points of $\C P_4$ can be mapped to any other given $6$ points of $\C P_4$, by an element of $PGL(5,\C)$. In the case of $7$ points, we may pass projective $3$-planes through the quadruple of points, $p_1,p_2,p_3,p_4$ giving the $3$-plane denoted $1234$, resp. $p_1,p_2,p_3,p_5$ giving the $3$-plane denoted $1235$, resp. $p_1,p_2,p_3,p_6$ giving the $3$-plane denoted $1236$, $p_1,p_2,p_3,p_7$ giving the $3$-plane denoted $1237$, see \cite[p. 2, Th. I]{sempleroth}. These are member of a pencil (as, see \cite[p.~2, Th. III]{sempleroth}, there is a unique pencil of $3$-planes with given base the projective plane generated by $p_1$, $p_2$, $p_3$), and thus have a cross-ratio (recall that a pencil of hyperplanes is the same as a projective line in the dual projective space), which is, of course an absolute invariant, under the simultaneous action of $PGL(5,\C)$ on the points. For more on invariants of points sets under the projective group we refer to \cite{conte, olver2}.
 
Rather than general invariants of a general group of points of $\C P_4$ under $PGL(5,\C)$, we want instead to find the invariants of points situated on the cone $\mathscr S_3$ with respect to the group $G$. Here invariants arise for four points. Indeed these determine four generators of the cone $\mathscr S_3$, which have a cross-ratio, namely, the cross-ratio of the four points in which these generators intersect anyone of the twisted cubics of $\mathscr S_3$. This invariant is equivalent to $r_4$. To see this we recall that the dimension of the family of twisted cubic on the cone $\mathscr S_3$ is four (see the discussion in \cite[p.~7]{wone}). We only look at what happens in the case of the twisted cubic lying on the plane $z_5=0$, which we denote by $c_3$ for simplicity. The other cases are similar. Let us choose (see \eqref{lfqhfzhlzdlh}) four generic points $p_1=\left[1:v_1:v^2_1:v^3_1:w_1\right]$, $p_2=\left[1:v_2:v^2_2:v^3_2:w_2\right]$, $p_3=\left[1:v_3:v^2_3:v^3_3:w_3\right]$ and $p_4=\left[1:v_4:v^2_4:v^3_4:w_4\right]$. They give the following four distinct points of the twisted cubic on the plane $z_5=0$, namely $\left[1:v_1:v^2_1:v^3_1:0\right]$, $\left[1:v_2:v^2_1:v^3_2:0\right]$, $\left[1:v_3:v^2_3:v^3_3:0\right]$, $\left[1:v_4:v^2_4:v^3_4:0\right]$. These four points of the considered twisted cubic have a "cross ratio" $R(p_1,p_2,p_3,p_5)$, \cite[p.~301, Cor. 4]{semplekneebone}. 

On the plane section $z_5=0$, the action of the group $G$ on the twisted cubic, boils down to the action of its subgroup given by
\begin{equation}
\label{fazmjmhfqlhk}
\begin{split}
& Z_1=d^3z_1+3cd^2z_2+3c^2dz_3+c^3z_4\\
 &Z_2=bd^2z_1+b(bc+2ad)z_2+a(ad+2bc)z_3+a^2cz_4\\
 &Z_3=b^2dz_1+d(ad+2bc)z_2+c(bc+2ad)z_3+ac^2z_4\\
 &Z_4=b^3z_1+3ab^2z_2+3a^2bz_3+a^3z_4,
 \end{split}
\end{equation}
which corresponds to the subgroup of $G$ described as follows
\begin{equation}
\label{mqmmdqjm}
\left[\left(\begin{array}{ccccc}d^3 & 3cd^2 & 3c^2d & c^3 & 0 \\bd^2 & b(bc+2ad) & a(ad+2bc) & a^2c & 0 \\b^2d & d(ad+2bc) & c(bc+2ad) & ac^2 & 0 \\b^3 & 3ab^2 & 3a^2b & a^3 & 0 \\0 & 0 & 0 & 0 & 1\end{array}\right)\right]\subset G\subset PGL(5,\C),
\end{equation}
 where $\left(\begin{array}{cc}c & d \\a & b\end{array}\right)\in GL(2,\C)$ and for $M\in GL(5,\C)$, $\left[M\right]$ means its class in $PGL(5,\C)$. One readily sees that \eqref{mqmmdqjm} is induced by the irreducible representation of $GL(2,\C)$ into $GL(4,S^3\C^2)$, where $S^3\C^2\simeq\C^4$ designates the third symmetric product of $\C^2$ with itself. Now one may show (see \cite[Th. 2]{wiltsche}, \cite[Examp. 10.9, 10.12]{harris1992}) that $c_3$ is canonically isomorphic to the projective line, and in this isomorphism, the automorphism group of $c_3$ is given by the class in $PGL(4,\C)$ of the image of the irreducible representation of $GL(2,\C)$ inside $GL(4,S^3\C^2)$, that is by the of classes of matrices given by
 
\begin{equation}
\label{mlpjohnokbb}
\left[\left(\begin{array}{cccc}d^3 & 3cd^2 & 3c^2d & c^3 \\bd^2 & b(bc+2ad) & a(ad+2bc) & a^2c  \\b^2d & d(ad+2bc) & c(bc+2ad) & ac^2  \\b^3 & 3ab^2 & 3a^2b & a^3 \end{array}\right)\right]\subset PGL(4,\C),
\end{equation}
 where $\left(\begin{array}{cc}c & d \\a & b\end{array}\right)\in GL(2,\C)$. While the complex projective line has automorphism group given by $PGL(2,\C)$, \cite[Th. 3]{wiltsche}. Since the cross ratio of four points on the projective line is invariant under the simultaneous action of $PGL(2,\C)$ on the points, it follows the "cross ratio" $R(p_1,p_2,p_3,p_4)$ of the four generically chosen points of the twisted cubic $c_3$ is invariant under the joint action of the group given by \eqref{mlpjohnokbb}, hence also with respect to $G$, since the points $p_1$, $p_2$, $p_3$, $p_4$ do not depend on the homogeneous coordinate $z_5$. Thus this gives an invariant equivalent to $r_4$.

\end{document}